\documentclass[a4paper,11pt]{article}

\title{Bent functions and line ovals}
\author{Kanat Abdukhalikov\thanks{This work was supported by UAEU grant 31S107}\\	
Department of Mathematical Sciences\\
UAE University, PO Box 15551, Al Ain, UAE\\
abdukhalik@uaeu.ac.ae}

\date{ }

\usepackage{amsthm,amsmath,amssymb}

\begin{document}

\maketitle

\newcommand{\Fn}{\mbox{$\mathbb{F}_{2^n}$}}                
\newcommand{\Fm}{\mbox{$\mathbb{F}_{2^m}$}}                
\newcommand{\Fnm}{\mbox{$\mathbb{F}_{2^m}^*$}}          
\newcommand{\Fmm}{\mbox{$\mathbb{F}_{2^m}^*$}}          
\newcommand{\Fpm}{\mbox{$\mathbb{F}_{p^m}$}}              
\newcommand{\Fp}{\mbox{$\mathbb{F}_{p}$}}                     
\newcommand{\Ft}{\mbox{$\mathbb{F}_{2}$}}                

\theoremstyle{plain}
\newtheorem{theorem}{Theorem}[section]
\newtheorem{lemma}[theorem]{Lemma}
\newtheorem{corollary}[theorem]{Corollary}
\newtheorem{proposition}[theorem]{Proposition}
\newtheorem{fact}[theorem]{Fact}
\newtheorem{observation}[theorem]{Observation}
\newtheorem{claim}[theorem]{Claim}

\theoremstyle{definition}
\newtheorem{definition}[theorem]{Definition}
\newtheorem{example}[theorem]{Example}
\newtheorem{conjecture}[theorem]{Conjecture}
\newtheorem{open}[theorem]{Open Problem}
\newtheorem{problem}[theorem]{Problem}
\newtheorem{question}[theorem]{Question}

\newtheorem{remark}[theorem]{Remark}
\newtheorem{note}[theorem]{Note}

\begin{abstract}
In this paper we study those bent functions which are linear on elements of  spreads,
their connections with ovals and line ovals, and we give descriptions of their dual bent functions.
In particular, we give a geometric characterization of Niho bent functions and of their duals, we 
give explicit formula for the dual bent function and present direct connections with ovals and line ovals.
We also show that bent functions which are linear on elements of inequivalent spreads can be EA-equivalent.
\end{abstract}

Keywords:  Spreads, ovals, line ovals, quasifields, semifields, bent functions, Niho bent functions. 

MSC: 51E15, 51E21, 51E23, 12K10, 94A60. 

\section{Introduction}

Bent functions were introduced by Rothaus \cite{Rothaus76}  and then they  were studied by Dillon \cite{Dillon74}
as Hadamard difference sets.  A bent function is a Boolean function with an even number of variables which
achieves the maximum possible distance from affine functions \cite{Carlet2010}.  Bent functions have relations
to coding theory, cryptography, sequences, combinatorics and designs theory \cite{Ab2015,Carlet2010,CarletMes2016}.

Dillon \cite{Dillon74} introduced bent functions related to partial spreads of
$\Fm \times \Fm$.
He constructed bent functions that are constant on the elements of a spread.
Dillon also studied a class of  bent functions that are linear on the
elements of a Desarguesian spread. These functions were throughly studied in
\cite{Buda2012,Buda2016,Carlet2011,Helleseth2012,Dobbertin2006,Leander2006} as Niho bent
functions.
In \cite{AbMes2015,Carlet2015,Ces2015,MesCM2015} these investigations were extended to
other types of spreads, and bent functions which are affine on the elements of spreads, were studied.

In this paper we study bent functions which are linear on the elements of  spreads
 and give geometric interpretations of their duals.
Carlet and Mesnager showed  \cite{Carlet2011} that 
any bent function which are linear on the elements of a Desarguesian spread
(they are equivalent to Niho bent functions in a bivariate form) determines an o-polynomial 
(oval polynomial) from finite geometry. Every o-polynomial defines an equivalence class of hyperovals, 
therefore Carlet and Mesnager revealed a general connection between Niho bent functions and hyperovals 
in Desarguesian spreads.  
But there are several inequivalent bent functions for each o-polynomial. 
We make result of Carlet and Mesnager more precise and show that bent functions linear on elements
of a Desarguesian spread  are in one-to-one correspondence with line ovals in an affine plane.
Points of the line oval completely define the dual bent function. 
More precisely, the zeros of the dual function of a Niho bent function are exactly the points of the line oval 
(in other words, the dual function of a Niho bent function is obtained from the characteristic
function of the set of points of the line oval by adding all-one constant function). 
Therefore, we have geometric characterization of Niho bent functions and of their duals.
In addition, starting from that line ovals one can construct ovals, but in general they will be in a projective plane,
not affine plane.
So every Niho bent function uniquely defines a line oval in an affine plane,  
and conversely, every line oval in an affine plane uniquely defines a Niho bent function.
Similarly, each Niho bent function uniquely defines an oval (in general, in a projective plane)
with a special property 
and conversely, such an oval uniquely determines a Niho bent function.
Hence, from known Niho bent functions one can get new compact representations of ovals,
and conversely from known representations of ovals one can get new examples of Niho bent functions. 
In particular, Subiaco and Adelaide hyperovals can be written in a very compact way, contrary to the complicated 
representation when they are written with the help of o-polynomials 
(Subiaco and Adelaide hyperovals have very complicated o-polynomials).  
In addition, our investigations allow us to address an open question on duals of Niho bent functions and
give explicit straightforward formula for the dual function of any Niho bent function.
We also discuss a question on EA-equivalence of Niho bent functions obtained from equivalent hyperovals.
Note that hyperovals have not yet been classified, and the list of known hyperovals can be found in  
\cite{Cher,Cher1996}.

Next we consider the general case of arbitrary spreads and  investigate  bent  functions
which are linear on the elements of a spread. 
We show that such bent functions are in one-to-one correspondence with line ovals in an appropriate affine plane. 
We also show that the dual of such bent function $f(x,y)$ can be characterized by the corresponding  line oval  
(the zeroes of the dual function $\tilde{f}(x,y)$ can be obtained from the points of the line oval by
swapping coordinates $(x,y)$). 
Then we study these constructions in two special cases related to symplectic spreads and consider examples.
We show that bent functions which are linear on the elements of inequivalent spreads
can be EA-equivalent by presenting such examples.

The paper is organized as follows.
We recall first in Section \ref{preliminary} definitions and notation concerning bent function,
spreads, ovals and line ovals.
In Section \ref{Niho} we study Niho bent functions, their duals and their connections with ovals and line ovals.
Next, in Section \ref{spreads} we consider the general case of spreads and  investigate  bent  functions
which are linear on the elements of a spread.


\section{Preliminary considerations and notation}
\label{preliminary}

We recall here some definitions and notation.

\subsection{Bent functions}

Let $K=\mathbb{F}_{2^n}$ and $\mathbb{F}_{2}$ be finite fields of orders
$2^n$ and $2$ respectively.
Let $\mathbb{F}^n_2$ be an $\mathbb F_2$-vector space of dimension $n$. We shall endow
$\mathbb{F}^n_2$ with the  structure of the field $\mathbb{F}_{2^n}$.
A Boolean function on $\mathbb{F}_{2^n}$  is a mapping from $\mathbb{F}_{2^n}$ to the prime field $\mathbb{F}_2$.

If $f$ is a Boolean function defined on $\mathbb{F}_{2^n}$, then the Walsh transform of $f$   is defined as follows:
$$  W_f(b) = \sum_{x\in \mathbb{F}_{2^n}} (-1)^{f(x)+ Tr(b x)}, $$
where $Tr$ is  the trace function from  $\mathbb{F}_{2^n}$ to $\mathbb{F}_2$.
Bent functions can be defined in terms of the Walsh transform.
Let $n$ be an even integer. A Boolean function $f$  on $\mathbb{F}_{2^n}$ is said to be {\em bent} if its Walsh transform satisfies $W_f(b) = \pm 2^{n/2}$ for all $b \in \mathbb{F}_{2^n}$.

Given a bent function $f$ over $\mathbb{F}_{2^n}$, we can always define its {\em dual function}, denoted by
$\tilde{f}$,  when considering the signs
of the values of the Walsh transform $W_f(b)$  of $f$. More precisely, $\tilde{f}$ is defined by the equation:
$$(-1)^{ \tilde{f}(x)}2^{n/2}=W_f(x). $$
The dual of a bent function is bent again, and  $\tilde{\tilde{f}}=f$.

Boolean functions $f,g: \mathbb{F}_{2^n}\rightarrow \mathbb{F}_2$ are \emph{extended-affine equivalent}
(in brief, EA-equivalent) if there exist an affine permutation $L$ of $\mathbb{F}_{2^n}$ and an affine function $\ell:\mathbb{F}_{2^n}\rightarrow \mathbb{F}_2$ such that $g(x)=(f\circ L)(x)+\ell(x)$.

If Boolean functions $f$ and $g$ are EA-equivalent and $f$ is bent then $g$ is bent too. 

\subsection{Ovals and line ovals}

Let $F=\mathbb{F}_{2^m}$ be a finite field of order  $q=2^m$. 
The affine plane $AG(2,q)$ has as points the vectors of $V=F\times F$,  and as lines 
$\{ (c,y) \mid y \in F\}$ and $\{ (x,xb +a) \mid x \in F\}$, $a, b, c \in F$.
These lines can be described by equations $x=c$ and $y=x b +a$.
The projective plane $PG(2,q)$ of order $q$ is obtained from $AG(2,q)$ by adding points at infinity 
in the following way. 
A set of parallel lines in $AG(2,q)$ define a point at infinity: the point at infinity corresponding to 
the parallel lines $x=c$, $c\in F$, is denoted by $(\infty)$ and the point at infinity corresponding to 
the parallel lines $y=xb+a$, $a\in F$, is denoted by $(b)$. 
Lines of $PG(2,q)$ are  $\{ (c,y) \mid y \in F\} \cup \{ (\infty) \}$ 
and $\{ (x,xb +a) \mid x \in F\}\cup \{ (b) \}$, $a, b, c \in F$, and the line at infinity $\{ (b) \mid b\in F \} \cup \{ (\infty) \}$.

Let $PG(2,q)$ be a finite projective plane of order $q$.
An \emph{oval} is a set of $q+1$ points, no three of which are collinear. Dually, a \emph{line oval} is a set of $q+1$
lines no three of which are concurrent. Any line of the plane meets the oval $\mathcal{O}$ at either
0, 1 or 2 points and is called exterior, tangent or secant, respectively. 
All the tangent lines to the oval $\mathcal{O}$ concur  \cite{Hir} at the same point $N$, called the \emph{nucleus} 
(or the knot) of $\mathcal{O}$. The set $\mathcal{O} \cup N$ becomes a \emph{hyperoval}, that is a set of $q+2$
points, no three of which are collinear. 
Conversely, by removing any point from hyperoval one gets an oval. 
If $\mathcal{O}$ is a line oval, then there is exactly one line $\ell$
such that on each of its points there is only one line of $\mathcal{O}$. This line is called the (dual)
\emph{nucleus} of $\mathcal{O}$. The $(q+2)$-set $\mathcal{O}\cup \{ \ell\}$ is \emph{line hyperoval} or
\emph{dual hyperoval}.

By a \emph{line oval} $\mathcal{O}$ in an affine plane $AG(2,q)$ we assume a set of $q+1$ lines in $AG(2,q)$, 
such that these lines, extended by corresponding points at infinity, determine a line oval in $PG(2,q)$ 
(whose nucleus is the line at infinity).  If a line $s$ belongs to a line oval $\mathcal{O}$, remaining $q$ lines from 
$\mathcal{O}$ intersect the line $s$ in $q$ different points. But the line $s$ contains $q$ points, therefore any point of $s$ 
belongs to two lines from $\mathcal{O}$.  In other words, for any line oval through any point of $AG(2,q)$ 
either 2 or 0 lines pass. 

For any oval in $PG(2,q)$ there are $q(q+1)/2$ secants and $q(q-1)/2$ exterior lines. 
Dually, let $\mathcal{O}$ be a line oval in an affine plane $AG(2,q)$ 
and $E(\mathcal{O})$ the set of points which are on the lines of the line oval $\mathcal{O}$:
$$E(\mathcal{O}) = \{ (x,y) \in V \mid (x,y) \ \rm{ is \ on \ a \ line \ of \ } \mathcal{O} \}.$$
Then each point of $E(\mathcal{O})$ belongs  to two lines of $\mathcal{O}$, 
$$|E(\mathcal{O})|=q(q+1)/2 $$ 
and there are $q(q-1)/2$ points in $AG(2,q)$ that do not belong to $\mathcal{O}$ 
(see, for example, \cite[p.~163, Corollary 1]{Hir}, \cite[p.~148, Statement 23]{Dem} or \cite[p.~126]{Mas2003}).  
Note that Kantor \cite[Theorem 7]{Kantor75} showed that $E(\mathcal{O})$ is a difference set in $K$.

\subsection{Polar representations}

Let $F=\mathbb{F}_{2^m}$ be a finite field of order  $2^m$, $q=2^m$, $n=2m$.
Consider $F$ as subfield of $K=\mathbb{F}_{2^n}$, so $K$ is a two dimensional vector space over $F$. 

As usually, for finite fields $\mathbb{F}_{2^{ks}} \supseteq \mathbb{F}_{2^k}$ the \emph{trace} function 
$Tr_{\mathbb{F}_{2^{ks}} / \mathbb{F}_{2^k}}$ of $x\in \mathbb{F}_{2^{ks}}$ over $\mathbb{F}_{2^k}$ is defined by 
$$Tr_{\mathbb{F}_{2^{ks}} / \mathbb{F}_{2^k}}(x)=x+x^{2^k}+\cdots +x^{2^{k(s-1)}}.$$
For different kinds of trace functions we denote
$$Tr (x)=Tr_{\Fn/ \Ft}(x), \quad T (x)=Tr_{\Fn / \Fm}(x), \quad tr(x)=Tr_{\Fm/ \Ft}(x).$$

The \emph{conjugate} of $x\in K$ over $F$ is
$$\bar{x} =x^q.$$
Then the \emph{trace map} from $K$ to $F$ is
$$Tr_{K/F}(x)=T(x) = x + \bar{x},$$
and the \emph{norm map} from $K$ to $F$ is 
$$N_{K/F} (x) = x \bar{x} = x^{1+q}.$$
The \emph{unit circle} of $K$ is the set of elements of norm $1$:
$$S= \{ u\in K : u\bar{u} =1 \}. $$
Therefore, $S$ is the multiplicative group of $(q+1)$st roots of unity in $K$.
Since $F\cap S = \{ 1\}$, each non-zero element of $K$ has  a unique polar coordinate representation
 $$ x=\lambda u$$
 with $\lambda\in F^*$ and $u \in S$. For any $x\in K^*$ we have
 $$\lambda = \sqrt{x \bar{x}},$$
 $$u= \sqrt{x/ \bar{x}}.$$
 
 In thesis \cite[p.~24]{Deorsey2015}, following ideas from \cite[p.~32]{Fisher2006}, the $\rho$-polynomials were
introduced in the following way. 
Since $K$ is a two dimensional vector space over $F$, one can 
identify points of affine plane $AG(2,q)$ with elements of the field $K=\mathbb{F}_{2^n}$. 
Consider hyperoval in $K$ containing $0\in K$. Then this hyperoval should contain exactly one point (not counting $0$)
on each line passing through $0$, i.e., on elements of the set $uF$,  for all $u \in S$.
Therefore, nonzero points of the hyperoval can be written as $u\rho(u)$, for some function $\rho: S \rightarrow F^*$.
Such functions are called $\rho$-polynomials in  \cite{Deorsey2015}.  Then the points of the hyperoval
will be $\{ u\rho(u) : u \in S \} \cup \{ 0\}$. In other words, $\{ u\rho(u) : u \in S \}$ is an oval with nucleus in $0\in K$. 

One can define nondegenerate bilinear form by $(\cdot , \cdot ) : K \times K \rightarrow F$ by
$$(x,y)=T(xy).$$
Lines of  $AG(2,q)$ can be considered as the zeroes of an equation $(a,x)+b=0$. Normalising  $a$ to $u\in S$, 
we see that lines of  $AG(2,q)$ can be considered  as the zeroes of an equation $T(ux)+\mu =0$:
$$L(u,\mu) =\{ x \in K : T(ux)+\mu =0 \},$$
where $u\in S$ and $\mu\in F$ (see for details  \cite[subsection 2.1]{Ball1999}).
Note that there are $(q+1)q =q^2+q$ such lines, which coincides with the total number of lines in $AG(2,q)$.
Lines $L(u,\lambda)$ and $L(u,\mu)$ are parallel.

 \subsection{Niho bent functions}

The affine plane $AG(2,q)$ has as points the vectors of the space $V = F \times F$.
A Desarguesian spread of $F \times F$ is a collection  of $q+1$  one-dimensional subspaces such that every
nonzero point of $F \times F$ lies in a unique subspace \cite{Dem}. 
Elements of $K=\mathbb{F}_{2^n}$ can be considered as points of affine plane $AG(2,q)$.
Then the set
$$\{ uF : u \in S \}$$
is a spread.  We consider bent Boolean functions $f: K \rightarrow \Ft$, which are $\Ft$-linear on each element 
$uF$ of the spread, and call them {\em Niho bent functions}.

A positive integer $d$ (always understood modulo $2^n-1$) is said to be a {\em Niho exponent} and
$x\rightarrow x^d$ is a {\em Niho power function} if the restriction of $x^d$ to $F$ is linear or,
in other words, $d\equiv 2^j \pmod{2^m-1}$ for some $j<n$.   As one considers $Tr(ax^d)$ with $a\in K$,
without loss of generality, one can assume that $d$ is in the normalized form, i.e., with $j=0$.
Then we have a unique representation $d=(2^m-1)s+1$ with $2\le s \le2^m$.
If some $s$ is written as a fraction, this has to be interpreted modulo $2^m+1$ (e.g., $1/2 =2^{m-1}+1$).
Following are some known bent functions consisting of one or more Niho exponents:

\begin{enumerate}
\item
Quadratic function $Tr(ax^{d_1})$ with $a\in K^*$, $d_1=(2^m-1)\frac{1}{2}+1$.

\item
Binomials of the form $Tr(\alpha_1 x^{d_1} + \alpha_2 x^{d_2})$, where $d_1=(2^m-1)\frac{1}{2}+1$,
$\alpha_1, \alpha_2 \in K^*$, $(\alpha_1 + \overline{\alpha_1})^2 = \alpha_2  \overline{\alpha_2}$.
The possible values of $d_2$ are \cite{Dobbertin2006}:
$$d_2=(2^m-1)3+1,$$
$$d_2=(2^m-1)\frac{1}{6}+1 \quad ({\rm taking} \ m \ {\rm even}).$$

\item
Take $1<r<m$ with $\gcd(r,m)=1$ and define
$$f(x)=Tr(a^2 x^{2^m+1} + (a+\bar{a}) \sum_{i=1}^{2^{r-1}-1} x^{d_i}),$$
where $2^r d_i=(2^m-1)i+2^r$ and $a\in K$, $a+\bar{a} \not= 0$ \cite{Leander2006}. 
In particular, in the simplest case that $r=2$ and $m$ is odd, one has 
$$f(x)=Tr(a x^{(2^m-1)\frac{1}{2}+1} + (a+\bar{a}) x^{(2^m-1)\frac{1}{4} +1}),$$
which is the case considered in \cite{Dobbertin2006}. 
\end{enumerate}


\section{Geometric characterization of Niho bent functions}
\label{Niho}

Let Boolean function $f: K \rightarrow \Ft$ be linear on elements of the spread $\{ uF : u \in S \}$.
Then  for any $\lambda \in F$, function $f$ can be defined by
\begin{equation}
\label{func}
f(\lambda u)=tr(\lambda g(u))
\end{equation}
for some function $g: S \rightarrow F$.

\begin{theorem}
\label{thm:main}
Let the function $f$ be defined by Equation (\ref{func}).
Then the following statements are equivalent:
\begin{enumerate}
\item
The function $f$  is  bent;
\item
The set $\mathcal{O} = \{ L(u,g(u)) : u\in S\}$  is a line oval in $K$. 
\end{enumerate}
In this case the dual bent function for $f(x)$ is
$$\tilde{f}(x) = 1 + \chi_{E(\mathcal{O})}(x) =
\left\{ \begin{array} {l}
0, \ {\rm if} \ x\in E(\mathcal{O}), \\ 1, \ {\rm if} \ x\not\in E(\mathcal{O}),
\end{array} \right. $$
where $E(\mathcal{O})$ is the set of points which are on the lines of the line oval $\mathcal{O}$.

In addition, any line oval $\mathcal{O}$ in $K$ can be written as $\mathcal{O} = \{ L(u,g(u)) : u\in S\}$ 
for some function $g: S \rightarrow F$, and it 
determines a bent function, whose restrictions to the elements of the spread $\{ uF : u \in S \}$ are linear.
\end{theorem}

\begin{proof}
The Walsh transform of the function $f(x)$ is
\begin{eqnarray*}
W_f(b)
&  = & \sum_{x\in K} (-1)^{f(x) + Tr( xb)}     \\
&  = & 1+ \sum_{\lambda\in F^*, \ u\in S} (-1)^{f(\lambda u) +Tr (\lambda ub)}     \\
&  = & 1+ \sum_{\lambda\in F^*, \ u\in S} (-1)^{tr(\lambda g(u)) + tr ( \lambda ub + \lambda \bar{u}\bar{b})}     \\
&  = & 1 -(q+1)+ \sum_{\lambda\in F, \ u\in S} (-1)^{tr(\lambda (g(u) +  ub +  \bar{u}\bar{b}))}     \\
&  = & -q+ \sum_{u\in S} \sum_{\lambda\in F}(-1)^{tr(\lambda (g(u) +  ub +  \bar{u}\bar{b}))}     \\
&  = & -q + | N_b| q = q( |N_b |-1),
\end{eqnarray*}
where $N_b = \{ u \in S :  \ g(u) +  ub +  \bar{u}\bar{b}=0\}$.
Therefore,  Boolean function $f(x)$  is  bent if and only if $|N_b|=2$ or $0$ for any $b$.
But
$$N_b = \{ u \in S :  \ g(u) +  ub +  \bar{u}\bar{b}=0\} = \{ u \in S :  \  b \in L(u,g(u)) \}.$$
Hence, Boolean function $f(x)$  is  bent if and only if there are 2 or 0 lines of the form $L(u,g(u))$ passing 
through any point $b\in K$, which means that  the set of $q+1$ lines 
$\mathcal{O} = \{ L(u,g(u)) : u\in S\}$ forms a line oval.

Futhermore, $|N_b|=2$ if and only if $b$ belongs to a line $L(u,g(u)) \in \mathcal{O}$.
Therefore,  $W_f(b)=q$ if and only if $b \in E(\mathcal{O})$, which means
$\tilde{f}(b) = 0$ if and only if  $b \in E(\mathcal{O})$.


Finally, let  $\mathcal{O}$ be a line oval in $K$. 
There are no parallel lines in $\mathcal{O}$, since for corresponding line oval in a projective plane 
two parallel lines and line at infinity are concurrent.
Therefore, the line oval $\mathcal{O}$ consists of $q+1$ different lines of the form $L(u,\mu_u)$, $u \in S$. 
Hence, for any $u\in S$ there is unique element $\mu_u\in F$, so the line oval $\mathcal{O}$ defines a function 
$g :S \rightarrow F$ given by  $g(u)=\mu_u$. 
By the first part of Theorem, Boolean function $f$ defined by Equation (\ref{func}) is bent.
\end{proof}

Therefore, bent function $f$ defined by Equation (\ref{func}) {\em canonically} corresponds to a line oval
$\mathcal{O} = \{ L(u,g(u)) : u\in S\}$ in $K$ and its dual function $\tilde{f}$ is determined by
characteristic function of $E(\mathcal{O})$. 

Theorem \ref{thm:main} gives a formula for dual bent functions.

\begin{corollary}
\label{dual}
Let $f$ be a Niho bent function and $f(\lambda u)=tr(\lambda g(u))$ for a function $g: S \rightarrow F$. 
Then the dual function for $f(x)$ is of the form 
$$\tilde{f}(x) = \prod_{u\in S} (T(xu) +g(u))^{q-1}.$$
\end{corollary}

\begin{proof} 
Let ${\mathcal O} = \{ L(u,g(u)) : u \in S\}$ be the line oval in $K$ corresponding to $f(x)$. 
By Theorem \ref{thm:main},  $\tilde{f}(x) = 0$ if and only if $x \in L(u,g(u))$ for some $u\in S$,
which means that $T(xu) +g(u) =0$ for some $u\in S$.
It is equivalent to $(T(xu) +g(u))^{q-1} =0$ since $T(xu) +g(u) \in F$.
Hence
$$\tilde{f}(x) = \prod_{u\in S} (T(xu) +g(u))^{q-1},$$
as required.
\end{proof}

Corollary \ref{dual} allows us to answer open question on duals of known Niho bent functions.
For instance, for the Niho bent function of the form $f(x)=Tr (a x^{d_1})$,
$d_1=(2^m-1)\frac{1}{2}+1$,  we have
$$f(\lambda u)=Tr (\lambda(a u^{d_1})) =
tr(\lambda(a u^{d_1} +\bar{a}\bar{u}^{d_1}))=
tr(\lambda (au^{(-2)\frac{1}{2}+1}+\bar{a}\bar{u}^{(-2)\frac{1}{2}+1}))=
tr(\lambda (a+\bar{a}))     ,$$
$$g(u)= a  +\bar{a} .$$
Therefore,
$$\tilde{f}(x) = \prod_{u\in S} (xu+\bar{x}\bar{u} +a + \bar{a})^{q-1}.$$

For the Niho bent function of the form $f(x)=Tr (a x^{d_1} +x^{d_2})$, $a+\bar{a}=1$,
$d_1=(2^m-1)\frac{1}{2}+1$, $d_2=(2^m-1)3+1$, we have
$$f(\lambda u)=Tr (\lambda(a u^{d_1} +u^{d_2})) =
tr(\lambda(a u^{d_1} +\bar{a}\bar{u}^{d_1}+u^{d_2} +\bar{u}^{d_2})),$$
\begin{eqnarray*}
g(u)
&=& a u^{d_1} +\bar{a}\bar{u}^{d_1}+u^{d_2} +\bar{u}^{d_2} \\
&=& au^{(2^m-1)\frac{1}{2}+1}+\bar{a}\bar{u}^{(2^m-1)\frac{1}{2}+1}+u^{(2^m-1)3+1}+ \bar{u}^{(2^m-1)3+1} \\
&=& au^{(-2)\frac{1}{2}+1}+\bar{a}\bar{u}^{(-2)\frac{1}{2}+1}+u^{(-2)3+1}+ \bar{u}^{(-2)3+1}\\
&=& 1+u^5+\bar{u}^5.
\end{eqnarray*}
Therefore,
\begin{equation}
\label{dual1}
\tilde{f}(x) = \prod_{u\in S} (xu+\bar{x}\bar{u} +1+u^5 + \bar{u}^5)^{q-1}.
\end{equation}

For the Niho bent function of the form $f(x)=Tr (a x^{d_1} +x^{d_2})$, $a+\bar{a}=1$,
$d_1=(2^m-1)\frac{1}{2}+1$, $d_2=(2^m-1)\frac{1}{6}+1$, we have
$$g(u)= 1+u^{(-2)\frac{1}{6}+1}+ \bar{u}^{(-2)\frac{1}{6}+1} = 1+u^{2/3}+\bar{u}^{2/3}, $$
\begin{equation}
\label{dual2}
\tilde{f}(x) = \prod_{u\in S} (xu+\bar{x}\bar{u} +1 +u^{2/3} + \bar{u}^{2/3})^{q-1}.
\end{equation}

Finally, for the Niho bent function of the form
$f(x)=Tr(a x^{2^m+1} + \sum_{i=1}^{2^{r-1}-1} x^{d_i})$,
where $2^r d_i=(2^m-1)i+2^r$ and $a\in K$, $a+\bar{a}= 1$, we have
\begin{eqnarray*}
g(u)
&=& 1+ \sum_{i=1}^{2^{r-1}-1} u^{d_i}+ \sum_{i=1}^{2^{r-1}-1} \bar{u}^{d_i} \\
&=& 1+ \sum_{i=1}^{2^{r-1}-1} u^{-2i/2^r +1}+\sum_{i=1}^{2^{r-1}-1} \bar{u}^{-2i/2^r +1}\\
&=& 1+ u^{-2/2^r +1}\cdot \frac{1-(u^{-2/2^r})^{2^{r-1}-1}}{1-u^{-2/2^r}} +
\bar{u}^{-2/2^r +1}\cdot \frac{1-(\bar{u}^{-2/2^r})^{2^{r-1}-1}}{1-\bar{u}^{-2/2^r}}  \\
&=& 1+ \frac{u^{-2/2^r +1}-1}{1-u^{-2/2^r}} + \frac{\bar{u}^{-2/2^r +1}-1}{1-\bar{u}^{-2/2^r}} \\
&=& 1+ \frac{u- u^{2/2^r}}{u^{2/2^r}-1} + \frac{\bar{u}- \bar{u}^{2/2^r}}{\bar{u}^{2/2^r}-1}  \\
&=& \frac{u + \bar{u} + u \bar{u}^{2/2^r} + \bar{u}u^{2/2^r}}{u^{2/2^r} + \bar{u}^{2/2^r}} \\
&=& \frac{u + \bar{u} + u^{2^{1-r}-1} + \bar{u}^{2^{1-r}-1}}{u^{2^{1-r}} + \bar{u}^{2^{1-r}}}
\end{eqnarray*}
for $u\not= 1$ and $g(1)=1$. Therefore, 
\begin{equation}
\label{dual-transl}
\tilde{f}(x) = (x+\bar{x} +1)^{q-1} \prod_{u\in S\setminus \{ 1\}} \left(xu+\bar{x}\bar{u} +
\frac{u + \bar{u} + u^{2^{1-r}-1} + \bar{u}^{2^{1-r}-1}}{u^{2^{1-r}} + \bar{u}^{2^{1-r}}}\right)^{q-1}.
\end{equation}

We note that the previous dual function $\tilde{f}(x)$ is also calculated in \cite{Buda2012} and given by: 
\begin{equation}
\label{dual-transl2}
\tilde{f}(x) = Tr ((e(1+x+\bar{x}) + e^{2^{n-r}} + \bar{x})(1+x+\bar{x})^{1/(2^r-1)} ) ,  
\end{equation}
where $e$ is an element from $K$ with property $e+\bar{e}=1$. 
Calculations in Magma \cite{Bosma} confirm that expressions (\ref{dual-transl}) and (\ref{dual-transl2}) 
represent the same function (moreover, calculations in Magma also confirm formulas (\ref{dual1}) and (\ref{dual2})). 
 
We also note that in the simplest case that $r=2$ and $m$ is odd, we have 
$g(u)= 1+ u^{\frac{1}{2}} + \bar{u}^{\frac{1}{2}}$ and formula (\ref{dual-transl}) becomes 
$$\tilde{f}(x) = \prod_{u\in S} \left(xu+\bar{x}\bar{u} + 1+ u^{\frac{1}{2}} + \bar{u}^{\frac{1}{2}}\right)^{q-1} .$$

Next we show that adding a linear function $Tr(cx)$ to $f(x)$ produces a shift of the corresponding
line oval by the element $c$. We remind that if $f(x)$ is bent then $f(x)+Tr(cx)$ is bent as well.
Consider shifting by an element $c$ on the affine plane $AG(2,q)$:
$$\tau_c : x \mapsto x+c.$$

\begin{proposition}
\label{shift}
Let a bent function $f$ be defined by $f(\lambda u)=tr(\lambda g(u))$, where $\lambda\in F$, $u\in S$, and
$\mathcal{O} = \{ L(u,g(u)) : u\in S\}$ be its corresponding line oval.
Define function $f_c(x) = f(x)+Tr(cx)$ and line oval
$\mathcal{O}_{c} = \tau_{c} \mathcal{O}$, where $c \in F$.
Then
$$\widetilde{f_c}(x)= 1+ {\chi}_{E(\mathcal{O}_{c})}.$$
\end{proposition}

{\bf Proof}. We have
\begin{eqnarray*}
W_{f_c}(b)
&  = & \sum_{x \in F} (-1)^{f(x)+Tr(cx) + Tr(bx)}     \\
&  = & W_f (b+c)   \\
& =  & \left\{ \begin{array} {l}
\ \  q, \ {\rm if} \ b+c \in E(\mathcal{O}), \\ -q, \ {\rm if} \ b+c \not\in E(\mathcal{O})
\end{array} \right. \\
& =  & \left\{ \begin{array} {l}
\ \  q, \ {\rm if} \ b\in E(\mathcal{O}_{c}), \\ -q, \ {\rm if} \ b\not\in E(\mathcal{O}_{c}).
\end{array} \right. \\
\end{eqnarray*}
Therefore,
$$\widetilde{f_c}(x)= 1+ {\chi}_{E(\mathcal{O}_{c})}. \quad \Box$$

We showed that Niho bent functions are in one-to-one correspondence with line ovals in an affine plane. 
Now, using duality between line ovals and ovals, we also show that Niho bent functions are in one-to-one 
correspondence with ovals with special property.  
Then known Niho bent functions allow us to get new compact representations of ovals. 
In particular, Subiaco and Adelaide hyperovals can be written in a very compact way, contrary to the complicated 
representation when they are written with the help of o-polynomials. 

Some ovals will be inside affine plane $AG(2,q)$ (identified with the field $K$), and some of them will be 
not in $AG(2,q)$, but in a projective plane $PG(2,q)$. We extend affine plane  $K$ to a projective plane
$\overline{K}=K \cup \{u_{\infty} : u \in S\}$, where $u_{\infty}$ is the point at infinity defined by 
the line $L(u,0)$.

\begin{lemma}
\label{lmm:lineoval}
The set of nonzero points $p_i\in K$, $i=1, 2, \dots, q+1$,  forms an oval if and only if
the set of lines $T(p_i x)=1$, $i=1, 2, \dots, q+1$, forms a line oval.
\end{lemma}

\begin{proof}
Let's show that if the set of nonzero points $\{ p_i : i=1, 2, \dots q+1\}$ is an oval then the set
$T(p_i x)=1$, $i=1, 2, \dots, q+1$, forms a line oval. Indeed, if  lines $T(p_i x)=1$, $T(p_j x)=1$, $T(p_k x)=1$
intersect in one point $b$, then
$$T(p_i b)=1, \quad T(p_j b)=1, \quad T(p_k b)=1.$$
Hence $T((p_i -p_j) b)=0$,  $T((p_j -p_k)b)=0$ and vectors $p_i-p_j$ and $p_j-p_k$ are collinear, so
points $p_i$, $p_j$ and $p_k$ are on one line.

Conversely, we show that if a set of lines $T(p_i x)=1$, $i=1, 2, \dots, q+1$, forms a line oval then the
points $p_i\in K$, $i=1, 2, \dots, q+1$,  forms an oval. Assume that points $p_i$, $p_j$ and $p_k$
are on one line.  Then vectors $p_i-p_j$ and $p_j-p_k$ are collinear,
so there is nonzero $b\in K$ such that $T((p_i -p_j) b)=0$,  $T((p_j -p_k)b)=0$.  Therefore,
$$T(p_i  b)= T(p_jb)= T(p_kb)= c$$
for some  $c\in F$. If $c=0$ then points $p_i$, $p_j$ and $p_k$ belong to one line $T(bx)=0$.
Therefore, lines $T(p_i x)=1$, $T(p_j x)=1$ and $T(p_k x)= 1$ are parallel and intersect in a point at infinity.
If $c\not= 0$  then
$$T(p_i  (b/c))=1, \quad T(p_j (b/c))=1, \quad  T(p_k (b/c))= 1,$$
hence $b/c$ belongs to three lines $T(p_i x)=1$, $T(p_j x)=1$ and $T(p_k x)= 1$.
\end{proof}

Let $\mathcal{O} = \{ L(u,g(u)) : u\in S\}$  be a line oval in $K$.
Note that $g(u)=0$ means that the line $L(u,g(u))$ passes through point $0$. Any point of $K$ can belong to 2 or 0
lines of a line oval. Hence if $g(u)=0$  for some $u\in S$, then there are exactly two solutions of the equation $g(u)=0$ with $u\in S$.
We consider the set  $\{\frac{u}{g(u)} : \ u\in S\}$, where we assume that  $\frac{u}{g(u)} = u_{\infty}$ if $g(u)=0$.

\begin{theorem}
\label{thm:oval}
Let a bent function $f$ be defined by Equation (\ref{func}).
\begin{enumerate}
\item
If $g(u)\not= 0$ for all $u\in S$ then the set $\{\frac{u}{g(u)} : \ u\in S\}$ forms an oval in $K$ with nucleus in $0$.

\item
There are $\frac{q(q-1)}{2}$ points  $c\in K$ such that  $g(u) +cu+\bar{c}\bar{u}\not= 0$ for all $u\in S$ and hence the set
$\{\frac{u}{g(u)+cu+\bar{c}\bar{u}}: \ u\in S\}$ forms an oval in
$K$ with nucleus in $0$.

\item
The set $\{\frac{u}{g(u)} : \ u\in S\} \cup \{ 0\}$ forms a hyperoval in $\overline{K}$.
\end{enumerate}
\end{theorem}

\begin{proof}
1) By Theorem \ref{thm:main} the set $\mathcal{O} = \{ L(u,g(u)) : u\in S\}$  is a line oval.
Therefore, the set of lines $T(\frac{u}{g(u)}x)=1$ forms line oval, and by Lemma \ref{lmm:lineoval} the set
$\{\frac{u}{g(u)}\}$ forms an oval.
We get an hyperoval by adding point $0$ to this oval, since on the line joining $0$ and $u$ there is only one
point from our oval.

2) Recall that $g(u)=0$ for some $u\in S$ if and only if the line $L(u,g(u))$ passes through the point $0$.
Since $E(\mathcal{O})$ does not cover all elements of $K$, we can choose $c\in K$ such that
$c\not\in E(\mathcal{O})$. There are  $\frac{q(q-1)}{2}$ such points  $c\in K$.
Therefore $c\not\in L(u,g(u))$ for any $u \in S$, which means $T(cu)+g(u)\not=0$
for any $u \in S$.   Consider bent function $f_c =f + Tr(cx)$.
Then $f_c$ is linear on elements of the spread $\{ uF : u \in S \}$ and its corresponding function
$g_c(u) = g(u)+ cu+\bar{c}\bar{u} = g(u) + T(cu) \not=0$ for any $u \in S$.
Now we can apply part 1) for the function $f_c$.

3) It follows from previous considerations. If $g(u)=0$ for some $u \in S$ then two points of hyperoval are
on the line at infinity. Then we can reason as in Lemma \ref{lmm:lineoval}.
\end{proof}

\begin{remark}
We note that, if $g(u)\not= 0$ for all $u\in S$, then the function $\frac{1}{g(u)}$ is a $\rho$-polynomial
in the sense of \cite{Deorsey2015}.
\end{remark}

\begin{corollary}
\label{cor:niho}
Let $f$ be a Niho bent function of the form $f(x)=Tr (a x^{d_1} +x^{d_2})$, $a+\bar{a}=1$, $d_1=(2^m-1)\frac{1}{2}+1$.
Then $g(u)=1+u^{d_2} +\bar{u}^{d_2}$.

1) If $g(u)\not= 0$ for all $u\in S$ then the set $\{\frac{u}{1+u^{d_2} +\bar{u}^{d_2}}: \ u\in S\}$ forms an oval in
$K$ with nucleus in $0$.

2) If for some $c\in K$ one has  $g(u)+cu + \bar{c}\bar{u}\not= 0$ for all $u\in S$ then the set
$\{\frac{u}{1+u^{d_2} +\bar{u}^{d_2}+cu+\bar{c}\bar{u}}: \ u\in S\}$ forms an oval in
$K$ with nucleus in $0$.
\end{corollary}
\begin{proof}
It follows from the previous theorem.
\end{proof}

Let $f$ be a Niho bent function of the form $f(x)=Tr (a x^{d_1})$, $a+\bar{a}=1$,
$d_1=(2^m-1)\frac{1}{2}+1$.
Then $g(u)=1$ and the set $\{ u: \ u\in S\}$ forms an oval in $K$ with nucleus in $0$.

Let $f$ be a Niho bent function of the form $f(x)=Tr (a x^{d_1} +x^{d_2})$, $a+\bar{a}=1$, $d_1=(2^m-1)\frac{1}{2}+1$,
$d_2=(2^m-1)3+1$.
Then
$$g(u)=1+u^5+\bar{u}^5.$$
It corresponds to a Subiaco hyperoval \cite{Deorsey2015,Helleseth2012}. 
Indeed, consider a $\rho$-polynomial
$$\rho(x) = \frac{x^5}{x^{10}+x^6+x^5+x^4+1}$$
from \cite{Deorsey2015} for a Subiaco hyperoval. Then define
$$g'(u)=\frac{1}{\rho(u)} = 1+u^5+\bar{u}^5 +u +\bar{u}. $$
This function  corresponds to the bent function $f(x)=Tr (a x^{d_1} +x^{d_2}+x)$.
Therefore, the set 
$$\left\{\frac{u}{1+u^5+\bar{u}^5 }: \ u\in S \right\} \cup \{ 0\}$$ 
forms a Subiaco hyperoval. 

Let $f$ be a Niho bent function of the form $f(x)=Tr (a x^{d_1} +x^{d_2})$, $a+\bar{a}=1$, $d_1=(2^m-1)\frac{1}{2}+1$,
$d_2=(2^m-1)\frac{1}{6}+1$,
$m$ even.
Then $g(u)=1+u^{2/3}+ \bar{u}^{2/3}$. For $u\not= 1$ we have  $g(u)= \frac{u+\bar{u}}{u^{1/3}+\bar{u}^{1/3}} \not= 0$.
Hence the set
$$\left\{\frac{u}{1+u^{2/3}+ \bar{u}^{2/3}}: \ u\in S \right\}$$
forms an oval in $K$ with nucleus in $0$.
It corresponds to an Adelaide hyperoval \cite{Deorsey2015,Helleseth2012}.
Indeed, consider a $\rho$-polynomial
$$\rho(x) = \frac{x(x^{\frac{1}{3}} + 1)^3}{(x+1)^3}, \quad \rho(1)=1$$
from \cite{Deorsey2015} for an Adelaide hyperoval. Then define
$$g'(u)=\frac{1}{\rho(u)} = \frac{(u+1)^3}{u(u^{\frac{1}{3}} + 1)^3} =
\frac{(u^{\frac{2}{3}} + u^{\frac{1}{3}} + 1)^3}{u} = 1+u^{2/3}+ \bar{u}^{2/3} +u +\bar{u}. $$
This function corresponds to the bent function $f(x)=Tr (a x^{d_1} +x^{d_2}+x)$.
Therefore, the set
$$\left\{\frac{u}{1+u^{2/3}+ \bar{u}^{2/3}}: \ u\in S \right\} \cup \{ 0\}$$
gives new representation of an Adelaide hyperoval in $K$.

We showed that Niho bent functions define ovals. Now we show that conversely from any oval in
$K$ one can construct Niho bent function.

\begin{theorem}
\label{thm:oval-bent}
Let ${\mathcal O}$ be an oval in $K$ with nucleus in $0$. Let $f(x)=tr(x/v)$ for $x\in vF$, where $v\in {\mathcal O}$.
Then $f(x)$ is a Niho bent function and
$$f(x)= \sum_{v\in {\mathcal O}} [(x^{q^2-q} - v^{q^2-q})^{q^2-1} +1] \sum_{j=0}^{m-1} (x/v)^{2^j}.$$
\end{theorem}

(In other words, if ${\mathcal O}$ is an oval in $K$ with nucleus in $0$ then the function defined by 
$f(x)=tr(\lambda)$ for $x=\lambda v$, $v \in {\mathcal O}$, is a Niho bent function.  
Note that if $x/v \in F$ then  $\sum_{j=0}^{m-1} (x/v)^{2^j}$ is the trace function $tr(x/v)$.)

\begin{proof}
Let ${\mathcal O}$ be an oval in $K$ with nucleus in $0$. Then any element of $v \in {\mathcal O}$ can be
written as $v=u\rho(u)=\frac{u}{g(u)}$ for some $\rho$-polynomial $\rho$ and corresponding function $g(u)$.
Let $x=\lambda u$, $u\in S$, $\lambda\in F$. Then  function
$$f(x)=tr(\lambda g(u)) = tr \left(\frac{x}{u} \cdot \frac{u}{v} \right) = tr\left(\frac{x}{v}\right)$$
is bent.

Now we should produce a formula for the function $f(x)$. If $v=\frac{u}{g(u)}$ then $\bar{v}^{-1} = u g(u)$.
Hence $u= \sqrt{v\bar{v}^{-1}} = \sqrt{v^{1-q}}= \sqrt{v^{q^2-q}}$.
On the other hand, if $x=\lambda u' \not= 0$ then $u'= \sqrt{x/\bar{x}} = \sqrt{x^{1-q}}= \sqrt{x^{q^2-q}}$.
So $u=u'$ if and only if $v^{q^2-q}=x^{q^2-q}$.
Therefore, a function which is equal to $tr(x/v)$ for $x\in vF$, and zero otherwise, can be written as
$$[(x^{q^2-q} - v^{q^2-q})^{q^2-1} +1] \sum_{j=0}^{m-1} (x/v)^{2^j}.$$
We note that $f(x)$ is a sum of such functions for $v \in {\mathcal O}$.
\end{proof}

\begin{example}
\label{ex:fisher}
Fisher and  Schmidt \cite{Fisher2006} showed the set $\{ u+u^3+u^{-3} : u\in S\} \cup \{ 0\}$ forms a hyperoval
in $K$. It is the Payne hyperoval when $m$ is odd, and the Adelaide hyperoval when $m$ is even.
Therefore, this hyperoval determines a new bent function
$$f(x)= \sum_{u\in S} [(x^{q^2-q} - (u+u^3+u^{-3})^{q^2-q})^{q^2-1} +1] \sum_{j=0}^{m-1} (x/(u+u^3+u^{-3}))^{2^j}.$$
\end{example}


\section{Bent functions linear on elements of spreads}
\label{spreads}

In this section we study bent functions which are linear on the elements of an arbitrary spread. 
The main result of the section is Theorem \ref{main2} where bent functions linear on the elements of a spread 
are shown to be in one-to-one correspondence with line ovals in an appropriate affine plane. 
First we recall some notation and facts on spreads.

\subsection{Notation and facts}
\label{notations}

Let $F=\mathbb{F}_{2^m}$ and consider $F \times F$ as a $2m$-dimensional vector space over $\Ft$.
We recall that a spread of $F \times F$ is a family of $2^m+1$ subspaces of dimension
$m$ such that every nonzero point of $F \times F$ lies in a unique subspace.
Every spread can be obtained from a, usually not unique, quasifield \cite{Dem}.

\begin{definition}
A system $Q=(V, +,\circ)$, with $|V|$ finite, is a (right) prequasifield if the following axioms hold:

(1) $(V,+)$ is an Abelian group, with additivity identity $0$;

(2) $(V^*, \circ )$ is a quasigroup;

(3) $(x+y)\circ z=x\circ z+y\circ z$ for all $x$, $y$, $z\in V$;

(4) $x\circ 0=0$ for all $x\in V$. \\
Prequasifield is a quasifield if it has a multiplicative identity. 
\end{definition}
Similarly one can define left prequasifield,
where left distributivity satisfies in place of right distributivity. Right prequasifield is a presemifield if it is
also a left prequasifield.

Any finite prequasifield can be obtained from a finite field $F$ preserving the addition $+$ and defining
new operation $\circ$, so we denote this prequasifield by $Q=(F,+,\circ )$.

The kernel $K(Q)$ of quasifield $Q=(F,+,\circ )$ consists of all the elements $k\in Q$ such that
$k\circ (x\circ y)= (k\circ x) \circ y$ and $k\circ (x+y) = k\circ x + k\circ y$ for all $x, y \in Q$.
The kernel $K(Q)$ is a field and $Q$ is a vector space over $K(Q)$.

Let  $tr$ be the absolute trace function from $F$ into $\Ft$.
We define a $\mathbb{F}_2$-bilinear form $B: F \times F \rightarrow \mathbb{F}_2$ by
$$B(x,y)=tr(xy),$$
and an alternating form on  $(F \times F) \times (F \times F)$ by
$$\langle (x,y),(x',y') \rangle = B(x,y') - B(y,x').$$

If $L: F \rightarrow  F$ is a $\mathbb{F}_2$-linear map, its adjoint operator $L^*$ with respect to
the form $B$ is defined as a unique linear operator satisfying the following condition:
$$B(L^*(x),y)=B(x,L(y)), \quad {\rm for \ all \ } x,y \in F.$$

Let $Q=(F, +,\circ)$ be a finite prequasifield.
With $Q$ one can associate a spread $\Sigma(Q)$, consisting of subspaces
$\{ (0,y) \mid y \in Q\}$  and   $\{ (x, x\circ z) \mid x \in Q \}$, $z \in Q$.
By modifying the multiplication, the prequasifield can be turned in a quasifield in such a way
that the respective associated spreads are the same.
Let's find the spread $\Sigma(Q)^{\perp}$, which is orthogonal (dual) to the spread $\Sigma(Q)$
with respect to the alternating form $\langle \cdot , \cdot \rangle$.
Define $R_z(x) = x\circ z$.
Then the spread $\Sigma(Q)^{\perp}$ can be written as spread, consisting of subspaces
$\{ (0,y) \mid y \in Q\}$  and   $\{ (x, R_z^*(x)) \mid x \in Q \}$, $z \in Q$, since
subspaces $\{ (x, R_z(x)) \mid x \in Q \}$ and $\{ (x, R_z^*(x)) \mid x \in Q \}$ are perpendicular.
Therefore, this spread is associated with other prequasifield $Q^t =(F,+, \star )$, which is called transpose
prequasifield and defined by  multiplication
$$x\star z =R_z^*(x).$$

In other words, $\Sigma(Q)^{\perp}=\Sigma(Q^t)$, where the  transpose prequasifield $Q^t =(F,+, \star )$ is defined by
\begin{equation}
\label{trans}
B(x \star z, y) = B(x,y\circ z).
\end{equation}
Therefore, relation (\ref{trans}) determines connection between prequasifield $Q =(F,+, \circ )$ and its
transpose $Q^t =(F,+, \star )$.

Let $Q=(F, +, \circ)$ be a  right prequasifield with respect to an operation $\circ$.
We define the dual left prequasifield $Q^d = (F, +, *)$  by operation
$$x * y = y\circ x.$$

Two presemifields $(S,+,\circ)$ and $(S^\prime,+,\circ^\prime)$
are said to be \emph{isotopic} if there exist three bijective linear mappings
$L$, $M$, $N$ from $S$ to $S^\prime $ such that
$$L(x\circ y) = M(x) \circ^\prime N(y), \ \forall x,y\in S.$$

In case of presemifields, using operations $S^d$ and $S^t$ one can get at most 6 isotopy classes of presemifields,
which is called the Knuth \cite{Knuth65,Lav2011} orbit  $\mathcal{K}(S)$ (or Knuth derivatives) of the presemifield $S$:
$$\mathcal{K}(S) = \{ [S],[S^d],[S^t],[S^{dt}],[S^{td}],[S^{dtd}]=[S^{tdt}]\}.$$

A presemifield $S=(F, +, \circ)$ is called commutative, if the operation $\circ$ of multiplication is commutative.
A prequasifield is called symplectic, if its associated spread is symplectic
(that is, every subspace from spread is isotropic with respect to the
alternating form $\langle \cdot , \cdot \rangle$).  This means
$$\langle (x,x \circ z),(y,y\circ z) \rangle = 0$$
for any $x, y, z \in F$. Equivalently,
\begin{equation}
\label{B}
B(x \circ z, y) = B(x,y \circ z)
\end{equation}
for any $x, y, z \in F$.

Equality (\ref{B}) means that all right multiplication mappings $R_z(x) = x\circ z$  of a symplectic prequasifield
are self-adjoint with respect to $B$.

A presemifield $S$ is commutative if and only if $S=S^d$, and
a presemifield $S$ is symplectic if and only if $S=S^t$.
Therefore, Knuth orbit of a commutative (symplectic) presemifield contains at most
three elements. If presemifield $S$ is commutative then $S^{td}$ is symplectic,
and if $S$ is symplectic  then $S^{dt}$ is commutative.
If presemifield $S$ is commutative then its transpose $S^{t}$ is dual to
symplectic presemifield $S^{td}$.

Starting from a symplectic presemifield $(F, +, \circ)$, one can construct a commutative presemifield
in the following way \cite{Kantor2003,Knarr}.
Consider the linear map $L_z : F \rightarrow F$, $L_z(x)=z\circ x$. Let $L_z^*$ be the adjoint operator
of $L_z$ with respect to the form $B$:
$$B(z\circ x, y) = B(L_z(x),y) = B(x, L_z^*(y)).$$
We introduce new operation $*$ by
$$z*y=L_z^*(y),$$
so
$$B(z \circ x, y) = B(x, z * y).$$
Then  $(F, +, *)$ is a commutative presemifield.
Similarly, starting from commutative presemifield $(F, +, *)$ and putting $L_z(x)=z*x$,
one can get a symplectic presemifield $(F, +, \circ)$:
$$B(z*x, y) =  B(L_z(x),y) = B(x, L_z^*(y)) = B(x, z\circ y).$$


Let $\Sigma$ be a spread of $2m$-dimensional vector space $V = F \oplus F$ over $\mathbb{F}_2$.
The affine plane $\mathcal{A}(\Sigma)$ has as points the vectors of $V$ and as lines the cosets $S+v$,
where $S\in \Sigma$ and $v\in V$.
If  spread $\Sigma(Q)$ is associated with a prequasifield $Q$, then affine plane
$\mathcal{A}(Q) = \mathcal{A}(\Sigma(Q))$ has as points the vectors of $V$ and as lines
$\{ (c,y) \mid y \in F\}$ and $\{ (x,x\circ b +a) \mid x \in F\}$, $a, b, c \in F$.
These lines can be described by equations $x=c$ and $y=x\circ b +a$.

Let $\mathcal{O}$ be a line oval in a affine plane $\mathcal{A}$ (so corresponding line oval in a projective 
plane has   the line at infinity as nucleus),
and $E(\mathcal{O})$ the set of points which are on the lines of the line oval $\mathcal{O}$:
$$E(\mathcal{O}) = \{ (x,y) \in V \mid (x,y) \ \rm{ is \ on \ a \ line \ of \ } \mathcal{O} \}.$$
Then each point of $E(\mathcal{O})$ belongs \cite{Hir} to two lines of $\mathcal{O}$ and
$$|E(\mathcal{O})|=q(q+1)/2 = 2^{2m-1}+2^{m-1}.$$

 Let $Q=(F,+,\circ )$ be a (right) prequasifield. We call mapping $G: F\rightarrow F$  an o-polynomial for
prequasifield $Q$ if $G$ is a permutation and the function $x\mapsto G(x) + x\circ b$ is $2$-to-$1$ function
for any nonzero  $b\in F$.

In the  affine plane $\mathcal{A} (Q)$, if $G(x)$ is an o-polynomial then the curve $y=G(x)$ intersects
with line  $y=x\circ b +a$ in one point if $b=0$, and $0 $ or $2$ points if $b \ne 0$.


The \emph {bivariate representation} of Boolean functions makes sense only when $n$ is an even integer,
which is the case for bent functions. For $n=2m$ we identify $\mathbb{F}_{2^n}$ with
$\mathbb{F}_{2^m}\times \mathbb{F}_{2^m}$ and consider the input to $f$ as an ordered pair $(x,y)$ of
elements of $\mathbb{F}_{2^m}$.
The function  $f$ being Boolean, its bivariate representation can be written in the (non unique) form
$f(x,y)=tr (P(x,y))$, where $P(x,y)$ is a polynomial in two variables over $\mathbb{F}_{2^m}$ and $tr$ is
the trace function from $\mathbb{F}_{2^m}$ to $\mathbb{F}_2$.
In this section we shall only consider functions in their bivariate representation.


\subsection{Spreads and bent functions}
\label{bent-spreads}

We recall the construction of bent functions from \cite{Carlet2015}.
Let $R_z : F \rightarrow F$ be a linear function for any $z \in F$.
Consider a collection of subspaces $\{ (x,R_z(x))  \mid x \in F \}$, $z\in F$,
and $\{ (0,y)  \mid y \in F \}$.
Let these subspaces be a spread, that is, the mapping $z \mapsto R_z(x)=y$ be a permutation of $F$
for any nonzero $x \in F$.
Denote by $\Gamma_x$ the inverse of this bijection, that is, $\Gamma_x(y) =z$.
A Boolean function on $F \times F$ is linear on the elements of the spread if and only if there exists
a function $G: F \rightarrow F$ and an element $\mu \in F$ such that, for every $y \in F$,
\begin{equation}
\label{mu}
f(0,y) = tr (\mu y),
\end{equation}
and for every $x, z \in F$,
\begin{equation}
\label{G}
f(x,R_z(x)) = tr (G(z)x).
\end{equation}

Up to EA-equivalence, one can assume that $\mu =0$. Indeed, one can add the linear function
$g(x,y)=tr(\mu y)$ to $f$; this changes $\mu$ into 0 and $G(z)$ into $G(z)+R_z^*(\mu)$,
where $R_z^*$ is the adjoint operator of $R_z$, since for $y=R_z(x)$ one has
$tr (\mu y) = B(\mu, y) = B(\mu, R_z(x)) = B(R_z^*(\mu),x) = tr (R_z^*(\mu)x)$.

We take $\mu =0$ in expression (\ref{mu}), and relation (\ref{G}) becomes
\begin{equation}
\label{GG}
f(x,y) = tr (G(z)x) = tr (G(\Gamma_x(y))x).
\end{equation}

\begin{theorem}[\cite{Carlet2015}, Theorem 2]
\label{car}
Consider a spread of $F \times F$ whose elements are $2^m$ subspaces of the form
$\{ (x,R_z(x))  \mid x \in F \}$, where, for every $z \in F$, function $R_z$ is linear, and
the subspace $\{ (0,y)  \mid y \in F \}$.
For every $x \in F^*$, let us denote by $\Gamma_x$ the inverse of the permutation
$z \mapsto R_z(x)=y$. A Boolean function defined by equation $(\ref{GG})$ is bent if and only if
$G$ is a permutation and, for every $b\neq 0$ the function $G(z) + R_z^*(b)$
is $2$-to-$1$, where $R_z^*$ is the adjoint operator of $R_z$.
\end{theorem}

An example of such function $G(x)$ was introduced in \cite{AbMes2015,Ces2015} in a particular case
of spreads related to symplectic semifields.

Let $D\subseteq F\times F$. We denote by $\chi_D(x,y)$ the characteristic function of $D$ and define
$$\bar{\chi}_D(x,y) = \chi_D(y,x) =
\left\{ \begin{array} {l}
1, \ {\rm if} \ (y,x)\in D, \\ 0, \ {\rm if} \ (y,x)\not\in D.
\end{array} \right. $$

Note that we changed the order of variables $x$ and $y$.
\begin{theorem}
\label{main}
Let $Q=(F, +, \circ)$ be a  prequasifield, $\Sigma(Q)$ be its associated spread,
and $Q^t=(F, +, \star)$ be its transpose prequasifield.
Let a Boolean function $f(x,y)$ be defined by equation $(\ref{GG})$.
Then the following statements are equivalent:

\begin{enumerate}
\item
A Boolean function defined by equation $(\ref{GG})$  is  bent;

\item
The function $G(z) + b\star z$ is  2-to-1 for all $b \in F^*$, and $G(z)$ is a bijection;


\item
$\mathcal{O} = \{ x=0\}  \cup \{ y=G(z) +x\star z \mid z \in F\}$ is a line oval for $\mathcal{A}(Q^t)$.
\end{enumerate}
In this case the dual bent function for $f(x,y)$ is
$$\tilde{f}(x,y) = 1 + \bar{\chi}_{E(\mathcal{O})}(x,y) = 1 + \chi_{E(\mathcal{O})}(y,x) =
\left\{ \begin{array} {l}
0, \ {\rm if} \ (y,x)\in E(\mathcal{O}), \\ 1, \ {\rm if} \ (y,x)\not\in E(\mathcal{O}),
\end{array} \right. $$
where $E(\mathcal{O})$ is the set of points which are on the lines of the line oval $\mathcal{O}$. Moreover,
$$\tilde{f}(x,y) = y^{q-1}\prod_{z \in Q} (y \star z + x + G(z))^{q-1}.$$
\end{theorem}

\begin{proof}
1) $\Leftrightarrow$ 2).
We put $R_z(x) = x\circ z$. The Walsh transform of the function $f(x,y)$ is
\begin{eqnarray*}
W_f(a,b)
&  = & \sum_{x, y \in F} (-1)^{f(x,y) + tr( ax+by)}     \\
&  = & \sum_{x, y \in F} (-1)^{tr(G(\Gamma_x(y))x + ax+by)}     \\
&  = & q\delta_0(b) + \sum_{x \in F^* , \ z \in F} (-1)^{tr(G(z)x + ax+bR_z(x))}     \\
&  = & q(\delta_0(b) -1) + \sum_{z \in F , \ x \in F} (-1)^{tr((G(z) + a+R_z^*(b))x)}     \\
&  = & q(\delta_0(b) -1 + | \{ z \in F , \ G(z) + a + R_z^*(b) = 0\} | ) \\
&  = & q(\delta_0(b) -1 + | \{ z \in F , \ G(z) + b\star z = a \} | ).
\end{eqnarray*}
Therefore,  a Boolean function $f(x,y) = tr (G(z)x)$  is  bent if and only if
the function $G(z) + b\star z$ is  2-to-1 for all $b \in F^*$, and $G(z)$ is a bijection.


2) $\Rightarrow$ 3).  Let the function $G(z) + b\star z$ be  2-to-1 for all $b \in F^*$, and $G(z)$  a bijection.
Consider a union of lines
$\mathcal{O} = \{ x=0\}  \cup \{ y=G(z) +x\star z \mid z \in F\}$.
Then the number of lines of $\mathcal{O}$ passing through point $(x,y)=(b,a)$
with   $b\ne 0$ is equal to the number of solutions of the equation $G(z) + b\star z=a$, which is 2 or 0.
The number of lines of $\mathcal{O}$ passing through point $(x,y)=(0,a)$ is 2.
Therefore, $\mathcal{O}$ is a line oval.

3) $\Rightarrow$ 2). Let $\mathcal{O} = \{ x=0\}  \cup \{ y=G(z) +x\star z \mid z \in F\}$ be a line oval
and let $E(\mathcal{O})$ be the set of points which are on the lines of $\mathcal{O}$.  Let $b\ne 0$.
If $(b,a) \in E(\mathcal{O})$  then there are two lines of $\mathcal{O}$ passing through $(b,a)$,
hence the equation $G(z) + b\star z=a$ has two solutions.
If $(b,a) \not\in E(\mathcal{O})$  then there are no lines of $\mathcal{O}$ passing through $(b,a)$,
hence the equation $G(z) + b\star z=a$ has no solution.
Finally, let $b=0$.
Then $(0,a) \in E(\mathcal{O})$  and there are two lines of $\mathcal{O}$ passing through $(0,a)$,
one of them is the line $x=0$, the second should be $y=G(z) +x\star z$ for some $z$,
hence the equation $a=G(z)$ has one solution.

Let conditions 1) - 3) be satisfied. Then the Walsh transform of the function $f(x,y)$ is
\begin{eqnarray*}
W_f(a,b)
& = & q(\delta_0(b) -1 + | \{ z \in F , \ G(z) + b\star z = a \} | ) \\
& =  & \left\{ \begin{array} {l}
\ \  q, \ {\rm if} \ (b,a)\in E(\mathcal{O}), \\ -q, \ {\rm if} \ (b,a)\not\in E(\mathcal{O})
\end{array} \right. \\
& = & q (-1)^{1 + \bar{\chi}_{E(\mathcal{O})}}.
\end{eqnarray*}

We have $\tilde{f}(x,y) = 0$ if and only if $(y,x) \in E(\mathcal{O})$, that is, if and only if
$x=G(z)+y\star z$ for some $z\in Q$ or $y=0$.
Hence
$$\tilde{f}(x,y) = y^{q-1}\prod_{z \in Q} (y \star z + x + G(z))^{q-1},$$
as required.
\end{proof}

Part 1) - 2) of the previous theorem was also proved in \cite{Ces2015}, and in case of Desarguesian spreads 
it is proved in \cite{Carlet2011}.  Carlet and Mesnager \cite[lemma 13]{Carlet2011} also showed that, in fact,  
the function $G(z)$ is an o-polynomial and revealed general connection between Niho bent functions 
(in bivariant presentation) and hyperovals. 
Every o-polynomial determines an equivalence class of hyperovals. 
However, there are several inequivalent bent functions for each o-polynomial. 
Recall that, if points of $PG(2,q)$ have coordinates $(x_0,x_1,x_2)$  and $G(z)$ is an o-polynomial, 
then corresponding hyperoval may be written, for example, as 
$$\mathcal{O} = \{ (1,z,G(z)) : z \in F \} \cup \{ (0,0,1), (0,1,0) \}.$$
We make result of Carlet and Mesnager more precise and show that bent functions linear on elements
of a Desarguesian spread  are in one-to-one correspondence with line ovals in an affine plane.
One of advantages of our approach is that it gives us a straightforward formula to calculate dual bent function. 
In addition, it becomes easier to study equivalence questions.

\begin{remark}
If  $Q$ is a semifield then the condition 2) in Theorem \ref{main} means that  $G(z)$ is an oval polynomial
for $Q^{td}$.
\end{remark}



We recall standard collineations of the affine plane ${\mathcal A} (Q)$:
$$\tau_{u,v} : (x,y) \mapsto (x+u,y+v),$$
$$\rho_c : (x,y) \mapsto (x,y+x\circ c),$$
$$\hat{\varphi} : (x,y) \mapsto (\varphi(x),\varphi(y)),$$
where $\varphi \in Aut(Q)$. These collineations generate the full group of collineations (automorphisms) of
${\mathcal A} (Q)$ for many types of semifields \cite{Hir}.
The spread  $\Sigma(Q)$ is invariant under the actions of collineations $\rho_c$ and $\hat{\varphi}$.


Collineations $\rho_c$ and $\hat{\varphi}$ induce functions
$$(\rho_c f)(x,y) = f(\rho_c^{-1}(x,y)) = f(x,y-x\circ c),$$
$$(\hat{\varphi} f)(x,y) = f(\varphi^{-1}(x),\varphi^{-1}(y)).$$

Next we show that adding a linear function $tr(ux+vy)$ to $f(x,y)$ produces a shift of the corresponding
line oval by the vector $(v,u)$. We remind that if $f(x,y)$ is bent then $f(x,y)+tr(ux+vy)$ is bent as well.

\begin{proposition}
\label{shift2}
Let $Q=(F, +, \circ)$ be a  prequasifield, $\Sigma(Q)$ be its associated spread, and
$Q^t=(F, +, \star)$ be the transpose prequasifield of $Q$.
Let a bent function $f(x,y)$ be defined by Equation $(\ref{GG})$, and
$\mathcal{O} = \{ x=0\}  \cup \{ y= x\star z +G(z) \mid z \in F\}$ be its corresponding line oval.
Define a function $f_{u,v}(x,y) = f(x,y)+tr(ux+vy)$ and a line oval
$\mathcal{O}_{v,u} = \tau_{v,u} \mathcal{O}$, where $u,v \in F$.
Then
$$\widetilde{f_{u,v}}(x,y)= 1+ \overline{\chi}_{E(\mathcal{O}_{v,u})}(x,y)=
1+\chi_{E(\mathcal{O}_{v,u})}(y,x).$$
\end{proposition}

\begin{proof}
We have
\begin{eqnarray*}
W_{f_{u,v}}(a,b)
&  = & \sum_{x, y \in F} (-1)^{f(x,y)+tr(ux+vy) + tr( ax+by)}     \\
&  = & W_f (a+u,b+v)   \\
& =  & \left\{ \begin{array} {l}
\ \  q, \ {\rm if} \ (b+v,a+u)\in E(\mathcal{O}), \\ -q, \ {\rm if} \ (b+v,a+u)\not\in E(\mathcal{O})
\end{array} \right. \\
& =  & \left\{ \begin{array} {l}
\ \  q, \ {\rm if} \ (b,a)\in E(\mathcal{O}_{v,u}), \\ -q, \ {\rm if} \ (b,a)\not\in E(\mathcal{O}_{v,u}).
\end{array} \right. \\
\end{eqnarray*}
Therefore,
$$\widetilde{f_{u,v}}(x,y)= 1+ \overline{\chi}_{E(\mathcal{O}_{v,u})}(x,y). $$
\end{proof}

The main result of this section is the following.

\begin{theorem}
\label{main2}
Let $Q=(F, +, \circ)$ be a  prequasifield, $\Sigma(Q)$ be its associated spread, and
$Q^t=(F, +, \star)$ be the transpose prequasifield of $Q$.
Then bent functions $f(x,y)$ which are linear on elements of the spread $\Sigma(Q)$,  are in
one-to-one correspondence with line ovals $\mathcal{O}$ in $\mathcal{A}(Q^t)$.
Dual function $\tilde{f}$ can be obtained from the characteristic function of $E(\mathcal{O})$ by
swapping coordinates $(x,y)$ and adding constant function $1$,
where $E(\mathcal{O})$ is the set of points which are on the lines of the line oval $\mathcal{O}$.
\end{theorem}

\begin{proof}
Let $f(x,y)$ be a bent function linear on the elements of the spread $\Sigma(Q)$.
Then $f(0,y)=tr(vy)$ for some $v\in F$.  Adding the linear function $tr(vy)$, we can assume that $f(0,y) =0$.
Then by Theorem \ref{main} it corresponds to a line oval $\mathcal{O}$ in $\mathcal{A}(Q^t)$. 
By Proposition \ref{shift2}, shifting line oval $\mathcal{O}$ by the vector $(v,0)$
gives a line oval for the original function $f(x,y)$.

Conversely, let  $\mathcal{O}$ be a line oval in $\mathcal{A}(Q^t)$. 
These lines can be described by equations $x=c$ and $y=x\circ z +G(z)$. There are no parallel lines in
$\mathcal{O}$, since two parallel lines and line at infinity are concurrent in a projective plane. 
Therefore, line oval $\mathcal{O}$
consists of $q$ different lines of the form $y=x\circ z +G(z)$ and one line $x=c$,
so $G(z)$ is a function from $F$ to $F$.
Shifting line oval by the vector $(c,0)$, we can assume that $c=0$.
Then this oval corresponds to a Boolean function
$f(x,y)$ defined by Equation $(\ref{GG})$.
\end{proof}

Therefore, bent functions $f(x,y)$, which are linear on the elements of a spread, are {\em canonically}
associated with line ovals in the translation plane of the dual (orthogonal) spread. 
Dual bent functions can be described by points on such line ovals.
These line ovals live in the affine plane $\mathcal{A} (Q^t)$ and the (dual) nucleus is 
exactly the line at infinity.
In the previous papers o-polynomials were considered, and people might try to associate them with ovals.
But we  see that  line ovals are more appropriate objects to study than ovals. 
And in addition, in case of semifields, one can get ovals in the projective plane associated with $Q^{td}$.

Now we study how collineations of  ${\mathcal A} (Q)$ may reflect line ovals and bent functions.
For the collineation $\rho_c : (x,y) \mapsto (x,y+x\circ c)$ of ${\mathcal A} (Q)$ we define
the corresponding collineation $\rho_c^t : (x,y) \mapsto (x,y+x\star c)$ of ${\mathcal A} (Q^t)$.

\begin{proposition}
\label{shift-2}
Let $Q=(F, +, \circ)$ be a  prequasifield, $\Sigma(Q)$ be its associated spread, and
$Q^t=(F, +, \star)$ be the transpose prequasifield of $Q$.
Let a bent function $f(x,y)$ be defined by Equation $(\ref{GG})$, and
$\mathcal{O} = \{ x=0\}  \cup \{ y= x\star z +G(z) \mid z \in F\}$ be its corresponding line oval.
Define a function $f_c = \rho_c f$ and a line oval
$\mathcal{O}_{c} = \rho_c^t(\mathcal{O})= \{ x=0\}  \cup \{ y= x\star z +G(z+c)\mid z \in F\}$.
Then $f_c$ is bent and
$$\widetilde{f_c}(x,y)= 1+ \overline{\chi}_{E(\mathcal{O}_{c})}(x,y)=1+\chi_{E(\mathcal{O}_{c})}(y,x).$$
\end{proposition}

\begin{proof}
We have
\begin{eqnarray*}
W_{f_c}(a,b)
&  = & \sum_{x, y \in F} (-1)^{(\rho_c f)(x,y) + tr( ax+by)}     \\
&  = & \sum_{x, y \in F} (-1)^{f(x,y+x\circ c) + tr( ax+by)}     \\
&  = & q\delta_0(b) + \sum_{x \in F^* , \ z \in F} (-1)^{tr(G(z+c)x + ax+bR_z(x))}     \\
&  = & q(\delta_0(b) -1) + \sum_{z \in F , \ x \in F} (-1)^{tr((G(z+c) + a+R_z^*(b))x)}     \\
&  = & q(\delta_0(b) -1 + | \{ z \in F , \ G(z+c) + a + R_z^*(b) = 0\} | ) \\
&  = & q(\delta_0(b) -1 + | \{ z \in F , \ G(z+c) + b\star z = a \} | )\\
& =  & \left\{ \begin{array} {l}
\ \  q, \ {\rm if} \ (b,a)\in E(\mathcal{O}_{c}), \\ -q, \ {\rm if} \ (b,a)\not\in E(\mathcal{O}_{c}).
\end{array} \right. \\
\end{eqnarray*}
Therefore,
$$\widetilde{f_c}(x,y)= 1+ \overline{\chi}_{E(\mathcal{O}_{c})}(x,y), $$
as required.
\end{proof}

\begin{corollary}
Up to EA-equivalence, in bent function defined by Equation $(\ref{GG})$ one can assume that $G(0)=0$.
\end{corollary}

\begin{proof}
Since $G$ is a permutation, there exists $c\in F$ such that $G(c)=0$.
Then for the bent function $f_c = \rho_c f$ its corresponding o-polynomial is $G_c(z)=G(z+c)$, and
 $G_c(0)=0$.
\end{proof}

\begin{proposition}
\label{aut}
Let $Q=(F, +, \circ)$ be a  prequasifield, $\Sigma(Q)$ be its associated spread, and
$Q^t=(F, +, \star)$ be the transpose prequasifield of $Q$.
Let a bent function $f(x,y)$ be defined by equation $(\ref{GG})$, and
$\mathcal{O} = \{ x=0\}  \cup \{ y= x\star z +G(z) \mid z \in F\}$ be its corresponding line oval.
For $\varphi \in Aut(Q)$ define function $f' = \hat{\varphi} f$ and line oval
${\mathcal O}_{\varphi} = \{ ((\varphi^*)^{-1}(x), (\varphi^*)^{-1}(y))  :  (x,y) \in {\mathcal O} \} $.
Then $f'$ is bent and
$$\widetilde{f'}(x,y)=
1+ \overline{\chi}_{E(\mathcal{O}_{\varphi})}(x,y)=1+\chi_{E(\mathcal{O}_{\varphi})}(y,x).$$
\end{proposition}

\begin{proof}
We have
\begin{eqnarray*}
W_{f'}(a,b)
&  = & \sum_{x, y \in F} (-1)^{f(\hat{\varphi}^{-1}(x,y)) + tr( ax+by)}     \\
&  = & \sum_{x, y \in F} (-1)^{f(\varphi^{-1}(x),\varphi^{-1}(y))) + tr(ax+by)}     \\
&  = & \sum_{x, y \in F} (-1)^{f(x,y) + tr( a \varphi(x)+b\varphi(y))}     \\
&  = & \sum_{x, y \in F} (-1)^{f(x,y) + tr( \varphi^*(a) x+ \varphi^*(b)y)}     \\
&  = & W_f (\varphi^*(a),\varphi^*(b))   \\
& =  & \left\{ \begin{array} {l}
\ \  q, \ {\rm if} \ (\varphi^*(b),\varphi^*(a))\in E(\mathcal{O}), \\
-q, \ {\rm if} \ (\varphi^*(b),\varphi^*(a))\not\in E(\mathcal{O}),
\end{array} \right. \\
& =  & \left\{ \begin{array} {l}
\ \  q, \ {\rm if} \ (b,a)\in E(\mathcal{O}_{\varphi}), \\ -q, \ {\rm if} \ (b,a)\not\in E(\mathcal{O}_{\varphi}).
\end{array} \right. \\
\end{eqnarray*}
Therefore,
$$\widetilde{f'}(x,y)= 1+ \overline{\chi}_{E(\mathcal{O}_{\varphi})}(x,y). $$
\end{proof}

\subsection{Bent functions related to symplectic spreads}
\label{semifields}

In this subsection we consider particular cases of spreads.

\begin{corollary} 
\label{symp}
Let $Q=(F, +, \circ)$ be a presemifield such that its transpose presemifield  $Q^t=(F, +, \star)$
is commutative (so $Q$ is dual to a symplectic presemifield).
Consider the spread  $\Sigma(Q)$. Let $G(z)=z\star z$ and let a Boolean function $f(x,y)$ be defined
by equation $(\ref{GG})$. Then function $f(x,y)$ is bent;
$\mathcal{O} = \{ x=0\}  \cup \{ y=z\star z +x\star z \mid z \in F\}$ is a line oval for
$\mathcal{A}(Q^t)$, and the dual bent function for $f(x,y)$ is
$$\tilde{f}(x,y) = 1 + \bar{\chi}_{E(\mathcal{O})}(x,y) = 1 + \chi_{E(\mathcal{O})}(y,x) =
\left\{ \begin{array} {l}
0, \ {\rm if} \ (y,x)\in E(\mathcal{O}), \\ 1, \ {\rm if} \ (y,x)\not\in E(\mathcal{O}),
\end{array} \right. $$
where $E(\mathcal{O})$ is the set of points which are on the lines of the line oval $\mathcal{O}$.
\end{corollary}

\begin{proof}
It is generally known that the function $G(z)=z\star z$ determines a hyperoval
for commutative semifield planes (see, for example, \cite{Jha}).
We provide a proof here for the sake of completeness.
Denote
$$H_b(z)=G(z) + b\star z = z\star z + b\star z = (z+b)\star z.$$
We note that $H_b(z)$ is a linear map over $\mathbb{F}_2$, since operation
$\star$ is commutative, and $\ker H_b = \{0,b\}$.
Therefore, the equation $H_b(z)=a$ has 0 or 2 solutions in $F$.

It remains to prove that $G$ is a permutation. Suppose that the linear map $G$ is not invertible.
Then there exists $a\in F^*$ such that $G(a)=0$.  Therefore $a\star a=0$, a contradiction.

Now the statement of corollary follows from Theorem \ref{main}.
\end{proof}

\medskip

Next we  consider spreads of symplectic quasifields and state a result similar to Corollary \ref{symp}.

\begin{corollary}
\label{symp2}
Let $Q=(F, +, \circ)$  be a symplectic prequasifield and $\Sigma(Q)$ be its associated spread.
Let $G(z)$ be an o-polynomial for the prequasifield $Q^d$ and
let a Boolean function $f(x,y)$ be defined by equation $(\ref{GG})$.
Then function $f(x,y)$ is bent;
$\mathcal{O} = \{ x=0\}  \cup \{ y=G(z) +x\circ z \mid z \in F\}$ is a line oval for $\mathcal{A}(Q)$,
and the dual bent function for $f(x,y)$ is
$$\tilde{f}(x,y) = 1 + \bar{\chi}_{E(\mathcal{O})}(x,y) = 1 + \chi_{E(\mathcal{O})}(y,x) =
\left\{ \begin{array} {l}
0, \ {\rm if} \ (y,x)\in E(\mathcal{O}), \\ 1, \ {\rm if} \ (y,x)\not\in E(\mathcal{O}),
\end{array} \right. $$
where $E(\mathcal{O})$ is the set of points which are on the lines of the line oval $\mathcal{O}$.
\end{corollary}

\begin{proof}
Immediately follows from Theorem \ref{main} using the fact that for symplectic
prequasifield we have $Q=Q^t$.
\end{proof}

The following theorem gives a lot of examples of symplectic prequasifields.

\begin{theorem}[\cite{Kantor2004}, Proposition 2.19]
Let $F=F_0 \supset F_1 \supset \cdots \supset F_n$ be a chain of distinct fields such that
$[F : F_n]$ is odd, with trace maps $T_i : F \rightarrow F_i$. Set $\lambda_0 = 1$;
let $\lambda_i \in  F_i^*$ and $\zeta_i \in F$ be arbitrary for $1\le i\le n$; and for $0\le i\le n$
write $c_i = \prod_{j=0}^i \lambda_j$. Define $Q=(F,+,\circ )$ by
\begin{eqnarray*}
x \circ y =  xy^2
& + & \sum^n_{i=1} [c_{i-1}yT_i(c_{i-1}xy) + c_iy T_i(c_i xy)] \\
& + & \sum^n_{i=1} [c_{i-1}yT_i(x\zeta_i) + \zeta_iy T_i(c_{i-1} xy)] .
\end{eqnarray*}
Then Q is a prequasifield coordinatizing a symplectic spread.
\end{theorem}

Using this Theorem and Corollary \ref{symp2} one can get new examples of bent functions.

Now we consider the case where we know examples of o-polynomials for symplectic prequasifields.
We recall one general example of line ovals for symplectic spreads \cite{Kantor75,Mas2003}
associated with prequasifield $Q=(F,+,\circ)$.
We assume that $F$ is an algebraic extension of a field $K$ of degree $n$. So $F$ can be considered
as $K$-vector space of dimension $n$. Let $T: F \rightarrow K$ be the trace map from $F$ into $K$.
We choose an orthogonal basis of $F$ over $K$ such that usual dot product $x \cdot y$ in $K^n$ is equal to $T(xy)$.
Assume that $K\subseteq K(Q)$ and $(kx)\circ y = k(x\circ y)$ for all $k\in K$ and $x\in F$.
It is clear that $R_z(x)= x\circ z$ is $K$-linear map.

We define the alternating bilinear form $\varphi : F\times F \rightarrow K$ by
$$\varphi((x_1,x_2),(y_1,y_2)) = T(x_1y_2 - x_2y_1).$$
Composition of $\varphi$ with absolute trace map $Tr_{K/\mathbb{F}_2}$ gives alternating bilinear form
into $\mathbb{F}_2$.

Let $Q$ be a prequasifield such that the spread $\Sigma(Q)$ of $K^n\times K^n$ be symplectic with respect
to form $\varphi$. Then subspaces of $\Sigma(Q)$ are given by
$$ \{  (x, x\circ z) \mid x\in K^n \} = \{ (x,xM_z) \mid x \in K^n \},$$
where $M_z$ is the matrix of right multiplication (in $Q$) by $z$. It is clear that $M_z$ is symmetric, since
the spread is symplectic:
$$0=\varphi((x,xM_z),(y,yM_z)) = x\cdot (yM_z) - xM_z \cdot y = x (yM_z)^t - xM_zy^t = x(M_z^t -M_z)y^t.$$

Define the map $d$ from the set of symmetric matrices into $K^n$, which associates to every
symmetric matrix $M$ the vector $d(M)$ whose components are the square roots of diagonal elements
of $M$ in their natural order. Then
$$\mathcal{O} = \{ x=0\}  \cup \{ y=d(M_z) +xM_z \mid z \in K^n\}$$
is a line oval \cite{Mas2003} in $\mathcal{A}(Q)$. We also denote $d(z) = d(M_z)$.

Define quadratic form $\mathbf{q}: F \times F \rightarrow \mathbb{F}_2$ by
$\mathbf{q}((x,y)) = Tr_{K/\mathbb{F}_2}(x\cdot y) = Tr_{K/\mathbb{F}_2}(T(xy))$ and denote by
$$S(\mathbf{q}) = \{ v \in F^2 \mid \mathbf{q}(v) =0 \}$$
the set of singular vectors of  $\mathbf{q}$ (including the zero vector). Then \cite{Kantor75,Mas2003} one has
$$E(\mathcal{O}) = S(\mathbf{q}) .$$

\begin{example}
\label{ex:des}
Consider Desarguesian plane $F \times F$, $F= \mathbb{F}_{2^m}$. The alternating bilinear form is
$\varphi((x_1,x_2),(y_1,y_2)) = tr(x_1y_2 - x_2y_1).$ The symplectic spread consists of
 the subspace $\{(0,y) \mid y\in F\}$ and subspaces $\{(x,xz) \mid x\in F\}$, $z \in F$.
We have $d(z)= \sqrt{z}$.

Then the corresponding line oval is
$$\mathcal{O} = \{ x=0\}  \cup \{ y= \sqrt{z} +xz \mid z \in F\}.$$
Therefore,
$$f(x,y) = tr(G(y/x)x) = tr(\sqrt{y/x} \ x) = tr(xy),$$
$$\mathbf{q}((x,y)) = tr(xy),$$
$$\tilde{f}(x,y)  = 1 + \bar{\chi}_{S(\mathbf{q})}(x,y)  = tr(xy),$$
as one could expect.
\end{example}

\begin{example}
\label{ex:lun}
The L\"{u}neburg symplectic spread is defined in the following way \cite{Luneburg}.
Let $F= \mathbb{F}_{2^m}$ be a finite field, $m=2k+1$.
Let $\sigma$ be the automorphism of $F$ defined by
$a^{\sigma} = a^{2^{k+1}}$. Then $a^{\sigma^2} = a^2$.
Points of the plane are $(x_1,x_2,y_1,y_2) \in F^4$.
Let $x = (x_1,x_2)$, $y = (y_1,y_2)$, $z = (z_1,z_2)$.
Define
$$M_z = \left( \begin{array}{cc}
z_1 & z_1^{\sigma^{-1}} + z_2^{1+\sigma^{-1} } \\
z_1^{\sigma^{-1}} + z_2^{1+\sigma^{-1} } & z_2
\end{array} \right),$$
$$x\circ z = x M_z =
(x_1 z_1 + x_2(z_1^{\sigma^{-1}} + z_2^{1+\sigma^{-1} }),
x_1(z_1^{\sigma^{-1}} + z_2^{1+\sigma^{-1} })+x_2 z_2).$$

Then the symplectic spread is defined by the subspace $\{(0,y) \mid y\in F^2\}$
and subspaces $\{(x,x\circ z) \mid x\in F^2\}$, $z \in F^2$.
One has
$$d(z)= (\sqrt{z_1},\sqrt{z_2}).$$

The corresponding line oval is
$$\mathcal{O} = \{ x=0\}  \cup \{ y= (\sqrt{z_1},\sqrt{z_2}) +xM_{(z_1,z_2)} \mid (z_1,z_2) \in F^2\}.$$
For quadratic form we have
$$\mathbf{q}((x_1,x_2,y_1,y_2)) = tr((x_1,x_2)\cdot (y_1,y_2)) = tr(x_1 y_1 + x_2 y_2).$$
Since $E(\mathcal{O}) = S(\mathbf{q})$, in this case the bent function is EA-equivalent to that from
Example \ref{ex:des}.
\end{example}

Therefore, we have the following interesting observation. Symplectic spreads from
Example \ref{ex:des} and Example \ref{ex:lun}
are not equivalent (they have different automorphism groups \cite{Luneburg}),
but they generate EA-equivalent bent functions.
Hence, bent functions which are linear on subspaces of inequivalent spreads can be EA-equivalent.
(In particular, this happens in the special case of line ovals which are
called completely regular line ovals and studied in \cite{Mas2003}, that was the case in
Examples \ref{ex:des} and \ref{ex:lun}.) In \cite{Ces2015} it was mentioned that such EA-equivalence would
give a surprising link between two inequivalent spreads.

\subsection{Desarguesian spreads}
\label{desarg}

In this subsection we consider the case of Desarguesian spreads.

\begin{proposition}
\label{des}
Let $\Sigma(F)$ be the Desarguesian  spread.
Let a Boolean function $f(x,y)$ be defined by equation $(\ref{GG})$, and
$\mathcal{O} = \{ x=0\}  \cup \{ y= xz +G(z) \mid z \in F\}$ be its corresponding line oval.
For $\psi \in GL(2,q) \langle \sigma \rangle$ define function $f' = \psi f$ and line oval
${\mathcal O}_{\psi} = \psi ({\mathcal O}) $.
Then $f'$ is bent and
$$\widetilde{f'}(x,y)= 1+ \overline{\chi}_{E(\mathcal{O}_{\psi})}(x,y).$$
\end{proposition}

\begin{proof}
Let $\psi \in GL(2,q)$ and
$$\psi : (x,y) \mapsto (x,y) \left( \begin{array}{cc} \alpha & \beta \\ \gamma & \delta  \end{array} \right) =
(\alpha x + \gamma y, \beta x + \delta y).$$
Assume first that $\det (\psi)=1$. Then we have
\begin{eqnarray*}
W_{f'}(a,b)
&  = & \sum_{x, y \in F} (-1)^{f(\delta x + \gamma y, \beta x + \alpha y) + tr( ax+by)}     \\
&  = & \sum_{x', y' \in F} (-1)^{f(x',y') + tr( a(\alpha x' + \gamma y') +b(\beta x'+\delta y'))}     \\
&  = & \sum_{x', y' \in F} (-1)^{f(x',y') + tr( (a \alpha  + b \beta )x' + (a\gamma + b\delta )y')}     \\
&  = & W_f (a \alpha  + b \beta , a\gamma + b\delta )   \\
& =  & \left\{ \begin{array} {l}
\ \  q, \ {\rm if} \ (a\gamma + b\delta , a \alpha  + b \beta)\in E(\mathcal{O}), \\
-q, \ {\rm if} \ (a\gamma + b\delta , a \alpha  + b \beta) \not\in E(\mathcal{O}),
\end{array} \right. \\
& =  & \left\{ \begin{array} {l}
\ \  q, \ {\rm if} \ (\psi^{-1}(b,a)) \in E(\mathcal{O}), \\
-q, \ {\rm if} \ (\psi^{-1}(b,a)) \not\in E(\mathcal{O}),
\end{array} \right. \\
& =  & \left\{ \begin{array} {l}
\ \  q, \ {\rm if} \ (b,a)\in E(\mathcal{O}_{\psi}), \\ -q, \ {\rm if} \ (b,a)\not\in E(\mathcal{O}_{\psi}).
\end{array} \right. \\
\end{eqnarray*}
Therefore,
$$\widetilde{f'}(x,y)= 1+ \overline{\chi}_{E(\mathcal{O}_{\psi})}(x,y). $$

For the map $\psi : (x,y) \mapsto (\lambda x, \lambda y)$ we can work similarly. Finally, if $\psi = \sigma$,
we note that $(\sigma^*)^{-1} = \sigma$ and we apply Proposition \ref{aut}.
\end{proof}


Now we consider question whether equivalent ovals (hyperovals) produce equivalent bent functions.
Extend affine plane $AG(2,q)$ to projective plane $PG(2,q)$.
We recall that automorphism (collineation) group of $PG(2,q)$ is
$P\Gamma L(3,q) = PGL(3,q)\langle\sigma\rangle$ and automorphism group of $AG(2,q)$ is
$A\Gamma L(2,q) = AGL(2,q)\langle\sigma\rangle$, where $\langle\sigma\rangle$ is the Galois group of $F$,
and $AGL(2,q) = F^2\cdot GL(2,q)$ is the affine group.
Consider an oval ${\mathcal O}$ in $PG(2,q)$.
By adding the nucleus $N$ to   ${\mathcal O}$ we get hyperoval ${\mathcal O}'$.
Conversely, taking a hyperoval ${\mathcal O}'$ and removing any point $N\in {\mathcal O}'$ we get an oval
${\mathcal O}'\setminus N$. Hence one hyperoval produces $q+2$ ovals. If ${\mathcal O}'$ is a hyperoval
and $N_1\in {\mathcal O}'$, $N_2\in {\mathcal O}'$, then we construct two ovals
${\mathcal O}_1={\mathcal O}'\setminus N_1$ and   ${\mathcal O}_2={\mathcal O}'\setminus N_2$.
Ovals ${\mathcal O}_1$ and ${\mathcal O}_2$ are (projectively) equivalent in $PG(2,q)$ if the stabilizer
of ${\mathcal O}'$ in $P\Gamma L(3,q)$ maps $N_1$ to $N_2$.  Therefore, the number of projectively
inequivalent ovals obtained from hyperoval ${\mathcal O}'$  is equal to the number of orbits of ${\mathcal O}'$
under the action of stabilizer of ${\mathcal O}'$. By duality, the same statement is true for dual ovals.

By Theorem \ref{main2} bent functions linear on the elements of spreads are in one-to-one correspondence
with line ovals with nucleus in the line at infinity.
Such line ovals are equivalent in $AG(2,q)$ under $A\Gamma L(2,q)$ if and only if
they are (projectively) equivalent in $PG(2,q)$ under $P\Gamma L(3,q)$,
since stabilizer in $P\Gamma L(3,q)$ of the line at infinity is equal to $A\Gamma L(2,q)$.
Line ovals are equivalent in $AG(2,q)$ under $A\Gamma L(2,q)$ if and only if  corresponding
bent functions are EA-equivalent. Hence bent functions are EA-equivalent if and only if
corresponding line ovals are projectively equivalent.
Therefore, the number of EA-inequivalent bent functions obtained from fixed hyperoval ${\mathcal O}'$
is equal to the number of orbits of ${\mathcal O}'$ under the action of automorphism group of ${\mathcal O}'$.

We note that similar results concerning EA-equivalence of Niho bent functions were announced without proof 
in  \cite{Penttila}, cited in \cite{CarletMes2016}.

\section{Conclusion}
\label{conclusion}

We considered  bent  functions which are linear on the elements of spreads.
We showed that duals of such bent functions can be characterized by line ovals.
We studied  these constructions for Desargusian spreads and in two special cases related to symplectic spreads.
In particular, we give a geometric characterization of Niho bent functions and their duals,
give explicit formulas for dual bent functions and present direct connections with ovals and line ovals.
Finally,  we showed that bent functions which are linear
on elements of inequivalent spreads can be EA-equivalent.

\bigskip

{\bf Acknowledgments}

\medskip

The author would like to thank Claude Carlet and Sihem Mesnager for
valuable discussions on the content of this paper.


\end{document}